\documentclass{article}
\usepackage[utf8]{inputenc}
\usepackage{amsmath}
\usepackage{listings, xcolor}
\usepackage{mathrsfs}
\usepackage{hyperref}
\usepackage{amsthm}
\usepackage{tikz}
\usepackage{comment}
\usepackage{diagbox}
\usepackage{longtable}
\usepackage{appendix}

    \newtheorem{theorem}{Theorem}[section]
    \newtheorem{cor}{Corollary}

    \theoremstyle{definition}
    \newtheorem{defi}{Definition}

    \theoremstyle{remark}

\title{Combinatorics of positional colored compositions}
\author{Andrew Li\footnote{\url{al5360@princeton.edu} Princeton University, Princeton, NJ 08544} and Hua Wang\footnote{\url{hwang@georgiasouthern.edu} Georgia Southern University, Statesboro, GA 30460}}
\date{}

\begin{document}

\maketitle

\begin{abstract}
    We consider colored compositions where only some parts are allowed different colors, depending on their locations in the composition. The counting sequences are obtained through generating functions. Connections to many other combinatorial objects are discussed, with combinatorial arguments provided and generalized for these observations. 
\end{abstract}

\noindent {\bf Keywords:} integer compositions, colored compositions, combinatorial proofs

\noindent {\bf 2010 MSC:} 05A17, 11B37

\section{Introduction}

Compositions are ordered sums of positive integers, which we call parts. 
For example, there are four compositions of 3. Namely,
\[1+1+1, \; 1+2, \; 2+1, \; 3\]

The $n$-color compositions, also known as colored compositions, are compositions where each part $k$ in the composition is given a color from 1 to $k$. We often denote the color of each part (of size $>1$) by a subscript. 
For example, there are eight colored compositions of 3. Namely,
\[1+1+1, \; 1+2_1, \; 1+2_2, \; 2_1+1, \; 2_2+1, \; 3_1, \; 3_2, \; 3_3\]

The $n$-color compositions were first introduced in \cite{Agarwal2000}. Since then, different variations of $n$-color compositions have been studied. 
In particular, spotted tilings \cite{Hopkins2012} are used to represent colored compositions, with the location of each spot denoting the color of each part. See, for example, Figure~\ref{fig:spotted}.

\begin{figure}[htbp]
\centering
    \begin{tikzpicture}[scale=.6]

        \draw [line width=1mm] (0,-4) -- (3,-4) -- (3, -3)-- (0, -3) -- cycle;
        \draw (1,-4) -- (1,-3);
        \draw (2,-4) -- (2,-3);

        \node[fill=black,circle,inner sep=3pt] at (0.5,-3.5) {};

        \draw [line width=1mm] (0,-4-2) -- (3,-4-2) -- (3, -3-2)-- (0, -3-2) -- cycle;
        \draw (1,-4-2) -- (1,-3-2);
        \draw (2,-4-2) -- (2,-3-2);

        \node[fill=black,circle,inner sep=3pt] at (2.5,-3.5-2) {};

     \end{tikzpicture}
\caption{The spotted tiling representation of $3_1$ (top) and $3_3$ (bottom).}
\label{fig:spotted}
\end{figure}
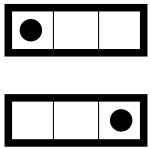

In \cite{Hopkins2021}, \cite{Acosta2019}, and \cite{Guo2012}, $n$-color compositions with  restricted colors are studied. In \cite{Dedrickson2012} and \cite{Collins2013}, various bijections involving $n$-color compositions as well as binary and ternary strings are presented. In \cite{Agarwal2008}, a bijection between $n$-color compositions and weighted lattice paths is proven. As related concepts, $n$-color cyclic compositions are studied in \cite{Gibson2018}, and palindromic $n$-color compositions are studied in \cite{Narang2006} and \cite{Guo2010}. 

In this paper, we consider colored compositions with only some parts allowing colors depending on their ``position'' in the composition. In Section~\ref{sec:count}, we present the basic definitions and the enumerative results through generating functions. Through the counting sequences, we found connections to many other combinatorial objects. Various bijections and some generalizations are discussed in Section~\ref{sec:com}. 

\section{Positional $n$-color compositions}
\label{sec:count}

The main concepts of interest in this paper are the so-called positional $n$-color compositions. 

\begin{defi}
    A $(m,k)$-$n$-colored composition of $\ell$ is a composition of $\ell$ such that the parts at positions $k \pmod m$ are $n$-colored while all other parts are not colored.
\end{defi}

For example, 
\[3_2 + 4 + 1 + 2_1 + 3 + 6 + 1_1\]
is a $(3,1)$-$n$-colored composition of $20$, where the parts in the first, fourth, and seventh positions (i.e., $1 \pmod 3$ positions) are $n$-colored.

As special cases, we refer to $(2,0)$-$n$-colored compositions as {\it EVEN colored compositions} and $(2,1)$-$n$-colored compositions as {\it ODD colored compositions}. We start with examining the generating functions of these compositions.

For each non-colored part, we have 
\[x+x^2+x^3+\cdots = \frac{x}{1-x} . \]

For each $n$-colored part, we have
\[x+2x^2+3x^3+\cdots = x (1+2x+3x^2 + \cdots ) = x \frac{d}{dx} \left( \frac{x}{1-x} \right) = \frac{x}{(1-x)^2} .\]

\subsection{EVEN colored compositions}
In such compositions, odd positioned parts are colorless and even positioned parts are $n$-colored. We consider two cases depending on the total number of parts.

\begin{itemize}
\item[{\it Odd:}]
When we have an odd number of parts, there is at least one non-colored part, followed by any nonnegative number of pairs of $n$-colored and non-colored parts. 
Each pair of colored and non-colored parts corresponds to the generating function
$$\frac{x}{1-x} \cdot \frac{x}{(1-x)^2} = \frac{x^2}{(1-x)^3} . $$
Hence the corresponding part of the generating function is
\begin{equation}\label{eq:genodd} \frac{x}{1-x}\sum_{i=0}^{\infty} \left(\frac{x^2}{(1-x)^3}\right)^i = \frac{x(1-x)^2}{(1-x)^3-x^2} .
\end{equation}

\item[{\it Even:}]
In this case, there are a positive number of pairs of colored and non-colored parts, resulting in the generating function
\begin{equation}\label{eq:geneven} 
\sum_{i=1}^{\infty}\left(\frac{x^2}{(1-x)^3} \right)^i= \frac{x^2}{(1-x)^3-x^2}.
\end{equation}
\end{itemize}

The sum of \eqref{eq:genodd} and \eqref{eq:geneven} gives us the generating function $F_{e}(x)$ of EVEN colored compositions
\begin{equation}\label{eq:genE}
\frac{x^3-x^2+x}{-x^3+2x^2-3x+1}.
\end{equation}
The corresponding counting sequence is listed as A034943 in the Online Encyclopedia of Integer Sequences: \url{https://oeis.org/A034943}. 

\subsection{ODD colored Compositions}
Through a similar fashion, we now consider compositions where odd positioned parts are $n$-colored and even positioned parts are colorless in two cases, depending on the number of parts:
\begin{itemize}
    \item[{\it Odd:}]
When we have an odd number of parts, there is at least one colored part, and then any non-negative number of pairs of colored and non-colored parts. 
\[\frac{x}{(1-x)^2}\sum_{i=0}^{\infty} \left(\frac{x^2}{(1-x)^3}\right)^i = \frac{x(1-x)}{(1-x)^3-x^2}\]
\item[{\it Even:}]
When we have an even number of parts, there are a positive number of pairs of colored and non-colored.
\[\sum_{i=1}^{\infty} \left(\frac{x^2}{(1-x)^3}\right)^i = \frac{x^2}{(1-x)^3-x^2} \]
\end{itemize}

Summing up these generating functions, we get
\[\frac{x}{-x^3+2x^2-3x+1}\]
The corresponding counting sequence is listed as A095263 in the Online Encyclopedia of Integer Sequences: \url{https://oeis.org/A095263}. 

\subsection{The $(m,k)$-$n$-colored compositions}
Through similar analysis, we now consider the more general $(m,k)$-$n$-colored compositions. We study the generating function in three cases depending on the number of parts: 
\begin{enumerate}
\item{$0 \pmod m$:} Every $m$ parts has 1 $n$-colored part and $m-1$ non-colored parts. Thus, each set of $m$ parts is represented by
\[\frac{x}{(1-x)^2}\left(\frac{x}{1-x}\right)^{m-1} = \frac{x^m}{(1-x)^{m+1}}.\]
Hence the corresponding part of the generating function is: 
\[\sum_{i=1}^{\infty} \left(\frac{x^m}{(1-x)^{m+1}}\right)^i = \frac{x^m}{(1-x)^{m+1} - x^m}.\]
\item{$j\pmod m$, where $1 \leq j \leq k-1$:} There are $j$ non-colored parts in addition to the previous case.  
Hence the corresponding part of the generating function is: 
\[\sum_{j = 1}^{k-1}\left(\frac{x}{1-x}\right)^j\sum_{i=0}^{\infty} \left(\frac{x^m}{(1-x)^{m+1}}\right)^i = \frac{\sum_{j=1}^{k-1}x^j(1-x)^{m-j+1}}{(1-x)^{m+1}-x^m}.\]
\item $\ell \pmod m$, where $k \leq \ell \leq m-1$: There are $\ell-1$ non-colored parts and one $n$-colored part in addition to those discussed in Case (1).  
Hence the corresponding generating function is: 
\[\sum_{\ell = k}^{m-1}\left(\frac{x^{\ell}}{(1-x)^{\ell+1}}\right)\sum_{i=0}^{\infty} \left(\frac{x^m}{(1-x)^{m+1}}\right)^i = \frac{\sum_{\ell = k}^{m-1}x^{\ell}(1-x)^{m-\ell}}{(1-x)^{m+1}-x^m}.\]
\end{enumerate}

We may obtain the generating function by summing these three cases. 


\section{Connections to other combinatorial objects}
\label{sec:com}

From the counting sequences, the EVEN and ODD colored compositions appear to have very interesting connections to a rich variety of combinatorial concepts. We explore these connections in this section. Many combinatorial proofs are given through new bijections, and some new connections are presented as generalizations.

\subsection{$n$-color compositions restricting color 2}
In Section 2.2 of \cite{Hopkins2021}, the recurrence relation 
\[a(n) = 3a(n-1)-a(n-2)-a(n-k)+a(n-k-d-1)\]
is presented, where $a(n)$ is the number of compositions of $n$ restricting colors $k, k+1, \hdots , k+d$. Letting $k = 2$ and $d= 0$, we have
\[a(n)= 3a(n-1)-2a(n-2)+a(n-3),\]
where $a(n)$ is the number of $n$-color compositions restricting the color 2, resulting in the same counting sequence as the EVEN colored compositions.

\begin{theorem}
For any positive integer $k$, the number of colored compositions of $k$ restricting the color 2 is the same as the number of EVEN colored compositions of $k$.     
\end{theorem}
\begin{proof}
We establish a bijection between the two groups of compositions by first removing the compositions that exist for both groups. These include compositions such as $(1, 1, 1, 1, \hdots, 1)$.

We now consider the compositions restricting the color 2 that are not EVEN colored compositions. That is, the compositions restricting color 2 that have one or more odd positioned parts with color greater than 1. 


In the spotted tiling representation of such compositions, we work from left to right. Whenever there is an odd positioned part with color greater than 1, say $p_c$ with $c\geq 3$, we cut it into two parts: $(c-2)+(p-c+2)_{2}$. Note that now the $n$-colored part $(p-c+2)_{2}$ becomes even positioned, and the odd positioned part $(c-2)$ is not $n$-colored. If an odd positioned part has color 1, then there is no change to that part other than removing its color. 


For the reverse mapping, we start with an EVEN colored composition. From left to right, for any even positioned parts of color 2, we merge that part with the part before it. That is, every odd positioned part $p$ followed by an even positioned part $q_2$ are combined to form a new odd positioned part $(p+q)_{p+2}$. 
\end{proof}

For example, the composition 
\[3_3, 1_1, 6_4, 4_4\]
would map to 
\[1, 2_2, 1,6_4,2, 2_2\]
through the following process. Starting with $3_3$, since it is in an odd position and its color exceeds 1, we split it into $1_1 + 2_2$. Now, the composition is $1, 2_2, 1_1, 6_4, 4_4$. Next, $1_1$ is in an odd position (3rd), but its color is 1, so it remains unchanged. The part $6_4$ is now in an even position, so no change is made. Finally, $4_4$ is in an odd position and its color exceeds 1, so it is split into $2 + 2_2$. The resulting composition is $1, 2_2, 1, 6_4, 2, 2_2$.

\subsection{$\binom{n}{2}$-Color Compositions}
In Section 4.1 of \cite{Dedrickson2012}, a bijection between $\binom{n}{2}$-color compositions of $k$ and 01- and 12-avoiding ternary strings of length $k-2$ is presented. As one can find in sequence A095263 of OEIS (\url{https://oeis.org/A095263}), the ODD colored compositions share the same counting sequence.
Next, we present the bijection between the ODD colored compositions and $\binom{n}{2}$-color compositions. We will make use of the spotted tiling representation of the latter, where each tile has two different spots (hence ${n \choose 2}$ ways to choose the color). Note: for spotted tilings of ODD colored compositions, even parts are considered to have color 1.

\begin{theorem}\label{thm:n2}
For any positive integer $k$, the number of ODD colored compositions of $k$ is equal to the number of $\binom{n}{2}$-color compositions of $k+1$. 
\end{theorem}
\begin{proof}
To define our map from ODD colored compositions of $k$ to the $\binom{n}{2}$-color compositions of $k+1$, we consider cases depending on the number of parts in the ODD colored compositions:

\begin{itemize}
    \item If there are an even number of parts in the ODD colored composition, we consider the spotted tiling representation. By combining every two adjacent tiles into one tile while keeping the spots at their original positions, we create a tiling representation with half the number of tiles and two spots in each tile. Expanding the last tile by one more cell (unspotted) gives us a tiling representation of $k+1$, where there are ${n \choose 2}$ ways to choose the two spots in each tile/part. Note that the last cell is always unspotted in this case. See Figure~\ref{fig3}. 

    \begin{figure}[htbp]
\centering
    \begin{tikzpicture}[scale=.6]

        \draw [line width=1mm] (0,-4) -- (16,-4) -- (16, -3)-- (0, -3) -- cycle;
        \draw (1,-4) -- (1,-3);
        \draw (2,-4) -- (2,-3);
        \draw [line width=1mm] (3,-4) -- (3,-3);
        \draw (4,-4) -- (4,-3);
        \draw (5,-4) -- (5,-3);
        \draw (6,-4) -- (6,-3);
        \draw [line width=1mm] (7,-4) -- (7,-3);
        \draw (8,-4) -- (8,-3);
        \draw (9,-4) -- (9,-3);
        \draw (10,-4) -- (10,-3);
        \draw (11,-4) -- (11,-3);
        \draw (12,-4) -- (12,-3);
        \draw [line width=1mm] (13,-4) -- (13,-3);
        \draw (14,-4) -- (14,-3);
        \draw (15,-4) -- (15,-3);

        \node[fill=black,circle,inner sep=3pt] at (2.5,-3.5) {};
        \node[fill=black,circle,inner sep=3pt] at (3.5,-3.5) {};
        \node[fill=black,circle,inner sep=3pt] at (11.5,-3.5) {};
        \node[fill=black,circle,inner sep=3pt] at (13.5,-3.5) {};

        \draw [line width=1mm] (0,-4-2) -- (17,-4-2) -- (17, -3-2)-- (0, -3-2) -- cycle;
        \draw (1,-4-2) -- (1,-3-2);
        \draw (2,-4-2) -- (2,-3-2);
        \draw (3,-4-2) -- (3,-3-2);
        \draw (4,-4-2) -- (4,-3-2);
        \draw (5,-4-2) -- (5,-3-2);
        \draw (6,-4-2) -- (6,-3-2);
        \draw [line width=1mm](7,-4-2) -- (7,-3-2);
        \draw (8,-4-2) -- (8,-3-2);
        \draw (9,-4-2) -- (9,-3-2);
        \draw (10,-4-2) -- (10,-3-2);
        \draw (11,-4-2) -- (11,-3-2);
        \draw (12,-4-2) -- (12,-3-2);
        \draw (13,-4-2) -- (13,-3-2);
        \draw (14,-4-2) -- (14,-3-2);
        \draw (15,-4-2) -- (15,-3-2);
        \draw (16,-4-2) -- (16,-3-2);

        \node[fill=black,circle,inner sep=3pt] at (2.5,-3.5-2) {};
        \node[fill=black,circle,inner sep=3pt] at (3.5,-3.5-2) {};
        \node[fill=black,circle,inner sep=3pt] at (11.5,-3.5-2) {};
        \node[fill=black,circle,inner sep=3pt] at (13.5,-3.5-2) {};

     \end{tikzpicture}
\caption{$\text{Mapping } 3_3+4_1+6_5+3_1=16 \text{ to } 7_{3,4}+10_{5,7}= 17$}
\label{fig3}
\end{figure}
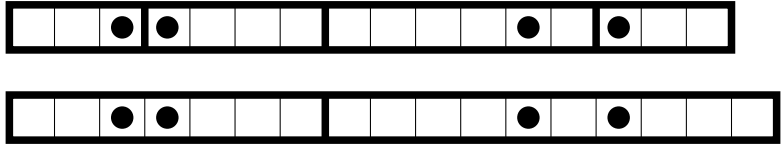

\item If there are an odd number of parts, we first append a single spotted cell/tile to the end of the spotted tiling representation, creating a spotted tiling of $k+1$ with an even number of parts. Then, combining every two adjacent tiles as described before yields a tiling representation of an ${n \choose 2}$-color composition of $k+1$. In this case, the last cell is always spotted. See Figure~\ref{fig4}.

\begin{figure}[htbp]
\centering
    \begin{tikzpicture}[scale=.6]

        \draw [line width=1mm] (0,-4) -- (15,-4) -- (15, -3)-- (0, -3) -- cycle;
        \draw (1,-4) -- (1,-3);
        \draw (2,-4) -- (2,-3);
        \draw (3,-4) -- (3,-3);
        \draw [line width=1mm] (4,-4) -- (4,-3);
        \draw (5,-4) -- (5,-3);
        \draw (6,-4) -- (6,-3);
        \draw [line width=1mm] (7,-4) -- (7,-3);
        \draw (8,-4) -- (8,-3);
        \draw (9,-4) -- (9,-3);
        \draw (10,-4) -- (10,-3);
        \draw (11,-4) -- (11,-3);
        \draw [line width=1mm] (12,-4) -- (12,-3);
        \draw (13,-4) -- (13,-3);
        \draw [line width=1mm] (14,-4) -- (14,-3);

        \node[fill=black,circle,inner sep=3pt] at (1.5,-3.5) {};
        \node[fill=black,circle,inner sep=3pt] at (4.5,-3.5) {};
        \node[fill=black,circle,inner sep=3pt] at (10.5,-3.5) {};
        \node[fill=black,circle,inner sep=3pt] at (12.5,-3.5) {};
        \node[fill=black,circle,inner sep=3pt] at (14.5,-3.5) {};

        \draw [line width=1mm] (0,-4-2) -- (16,-4-2) -- (16, -3-2)-- (0, -3-2) -- cycle;
        \draw (1,-4-2) -- (1,-3-2);
        \draw (2,-4-2) -- (2,-3-2);
        \draw (3,-4-2) -- (3,-3-2);
        \draw (4,-4-2) -- (4,-3-2);
        \draw (5,-4-2) -- (5,-3-2);
        \draw (6,-4-2) -- (6,-3-2);
        \draw [line width=1mm](7,-4-2) -- (7,-3-2);
        \draw (8,-4-2) -- (8,-3-2);
        \draw (9,-4-2) -- (9,-3-2);
        \draw (10,-4-2) -- (10,-3-2);
        \draw (11,-4-2) -- (11,-3-2);
        \draw (12,-4-2) -- (12,-3-2);
        \draw (13,-4-2) -- (13,-3-2);
        \draw [line width=1mm](14,-4-2) -- (14,-3-2);
        \draw (15,-4-2) -- (15,-3-2);

        \node[fill=black,circle,inner sep=3pt] at (1.5,-3.5-2) {};
        \node[fill=black,circle,inner sep=3pt] at (4.5,-3.5-2) {};
        \node[fill=black,circle,inner sep=3pt] at (10.5,-3.5-2) {};
        \node[fill=black,circle,inner sep=3pt] at (12.5,-3.5-2) {};
        \node[fill=black,circle,inner sep=3pt] at (14.5,-3.5-2) {};
        \node[fill=black,circle,inner sep=3pt] at (15.5,-3.5-2) {};

     \end{tikzpicture}
\caption{$\text{Mapping } 4_2+3_1+5_4+2_1+1_1=15 \text{ to } 7_{2,5}+7_{4,6}+2_{1,2}= 16$}
\label{fig4}
\end{figure}
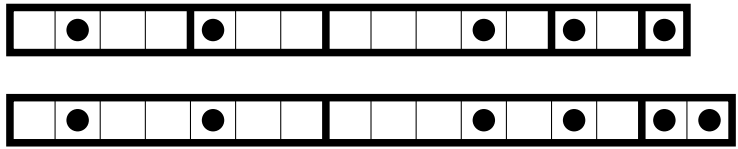
\end{itemize}

To define the inverse map, we cut every tile (with two spots) into two tiles immediately before the second spot. Note that this ensures that all even positioned parts have their spots at the beginning of the tile (hence colored 1, or uncolored). Lastly, we drop the last cell. Note that the last cell is only spotted if it is a tile by itself.
\end{proof}

As pointed out in OEIS, A095263 is also the counting sequence of the number of ways to split $k$ into an unspecified number of intervals and then choose 2 blocks from each interval. This is combinatorially explained through our bijection, because choosing two blocks in each interval is the same as choosing two spots for each part.

Furthermore, as a consequence of Theorem~\ref{thm:n2}, we have indirectly shown a bijection between odd-part n-color compositions of $k$ and the aforementioned ternary strings \cite{Dedrickson2012}. 

\begin{cor}\label{cor1}
For any positive integer $k$, the number of ODD colored compositions of $k$ is equal to the number of 01- and 12-avoiding ternary strings of length $k-1$.
\end{cor}


\subsection{Ternary Strings}
Corollary~\ref{cor1} can be generalized to a new bijection that is interesting on its own, where the EVEN colored compositions are mapped to a different type of ternary strings.

\begin{theorem}
The EVEN colored compositions of $k$ are equinumerous with the ternary strings of length $k$ which restrict consecutive digits, and do not begin with a 2 or end with a 0. 
\end{theorem}
\begin{proof}
We start with the spotted tiling representation of each composition, and consider each (vertical) line that separates adjacent square cells, including the end but not the beginning. They will be mapped to 0, 1, or 2. 
\begin{itemize}
\item If the line is completely inside an odd positioned part (with no spot), it is mapped to a 1. 
\item If the line is before the spot in an even positioned part (including the partition line that starts the part), it is mapped to a 0.  
\item If the line is after the spot in an even part (including the partition line that ends the part), it is mapped to a 2. 
\item If the tiling ends on an odd positioned part, then the final line is mapped to a 1.
\end{itemize}

See Figure~\ref{fig:tstring} for an illustration of the above process.
From the definition of this map, we see that:
\begin{itemize}
\item A 0 is always followed by a 0 or a 2.
\item A 1 is always followed by a 1 or a 0.
\item The composition cannot start with a 2 because a 0 must come before the first 2. 
\item The composition cannot end with a 0. 
\end{itemize}

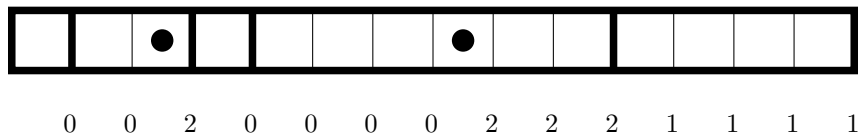
\begin{figure}[htbp]
\centering
    \begin{tikzpicture}[scale=.8]

        \draw [line width=1mm] (0,-4) -- (14,-4) -- (14, -3)-- (0, -3) -- cycle;
        \draw [line width=1mm](1,-4) -- (1,-3);
        \draw (2,-4) -- (2,-3);
        \draw [line width=1mm] (3,-4) -- (3,-3);
        \draw [line width=1mm] (4,-4) -- (4,-3);
        \draw (5,-4) -- (5,-3);
        \draw (6,-4) -- (6,-3);
        \draw (7,-4) -- (7,-3);
        \draw (8,-4) -- (8,-3);
        \draw (9,-4) -- (9,-3);
        \draw [line width=1mm] (10,-4) -- (10,-3);
        \draw (11,-4) -- (11,-3);
        \draw (12,-4) -- (12,-3);
        \draw (13,-4) -- (13,-3);

        \node[fill=black,circle,inner sep=3pt] at (2.5,-3.5) {};
        \node[fill=black,circle,inner sep=3pt] at (7.5,-3.5) {};

     \end{tikzpicture}
\[\hspace{1.15cm}0 \hspace{.625cm}0\hspace{.625cm} 2\hspace{.625cm} 0\hspace{.625cm} 0\hspace{.625cm} 0\hspace{.625cm} 0\hspace{.625cm} 2\hspace{.625cm} 2\hspace{.625cm} 2\hspace{.625cm} 1\hspace{.625cm} 1\hspace{.625cm} 1\hspace{.625cm} 1\hspace{.625cm}\]
\caption{$\text{Mapping } 1+2_2+1+6_4+4 = 14 \text{ to } 00200002221111$}
\label{fig:tstring}
\end{figure}

To define the inverse, we map every substring of consecutive 1's of length $m$ to an odd positioned part with size $m+1$. If there are no 1's between a 2 and a 0, or at the start of the string, an odd positioned part of size 1 is created. Every substring of consecutive 0's of length $a$ followed by a substring of consecutive 2's of length $b$ becomes an even positioned part of size $a+b-1$ with color $a$. 
\end{proof}

\subsection{Boolean Strings}

The following is observed from the sequence A034943 (\url{https://oeis.org/A034943}). We provide a combinatorial proof here.

\begin{theorem}
The number of EVEN colored compositions of $k+1$ is equal to the sum of the products of the lengths of the runs of 1's in all binary (boolean) strings of length $k$. 
\end{theorem}

\begin{proof}
To map a binary string to an EVEN colored composition, we first add a zero at the beginning. Then, we map a substring of consecutive 0's or 1's to a part of the corresponding size. For instance, starting from 001100, we have 0001100, resulting in the (regular) composition (3,2,2). 

Note that by putting a 0 at the front of each binary string, all the substrings of 0's correspond to odd positioned parts, and all the substrings of 1's correspond to even positioned parts. 

Now, each of the resulting (regular) compositions corresponds to multiple EVEN colored compositions, as the number of choices for the color of an even positioned part is exactly its size. Hence the total number of color choices is the product of the even positioned part sizes. Consequently, 
the number of EVEN colored compositions of $k+1$ is the same as the sum of the product of the lengths of the runs of 1's in each binary (boolean) string of length $k$. See Figure~\ref{fig:boolean} for an example. 

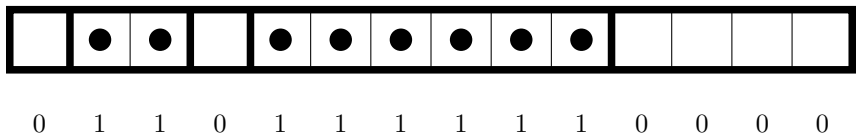
\begin{figure}[htbp]
\centering
    \begin{tikzpicture}[scale=.8]

        \draw [line width=1mm] (0,-4) -- (14,-4) -- (14, -3)-- (0, -3) -- cycle;
        \draw [line width=1mm](1,-4) -- (1,-3);
        \draw (2,-4) -- (2,-3);
        \draw [line width=1mm] (3,-4) -- (3,-3);
        \draw [line width=1mm] (4,-4) -- (4,-3);
        \draw (5,-4) -- (5,-3);
        \draw (6,-4) -- (6,-3);
        \draw (7,-4) -- (7,-3);
        \draw (8,-4) -- (8,-3);
        \draw (9,-4) -- (9,-3);
        \draw [line width=1mm] (10,-4) -- (10,-3);
        \draw (11,-4) -- (11,-3);
        \draw (12,-4) -- (12,-3);
        \draw (13,-4) -- (13,-3);

        \node[fill=black,circle,inner sep=3pt] at (1.5,-3.5) {};
        \node[fill=black,circle,inner sep=3pt] at (2.5,-3.5) {};
        \node[fill=black,circle,inner sep=3pt] at (4.5,-3.5) {};
        \node[fill=black,circle,inner sep=3pt] at (5.5,-3.5) {};
        \node[fill=black,circle,inner sep=3pt] at (6.5,-3.5) {};
        \node[fill=black,circle,inner sep=3pt] at (7.5,-3.5) {};
        \node[fill=black,circle,inner sep=3pt] at (8.5,-3.5) {};
        \node[fill=black,circle,inner sep=3pt] at (9.5,-3.5) {};
        
     \end{tikzpicture}
\[\hspace{.625cm} 0 \hspace{.625cm}1\hspace{.625cm} 1\hspace{.625cm} 0\hspace{.625cm} 1\hspace{.625cm} 1\hspace{.625cm} 1\hspace{.625cm} 1\hspace{.625cm} 1\hspace{.625cm} 1\hspace{.625cm} 0\hspace{.625cm} 0\hspace{.625cm} 0\hspace{.625cm} 0\hspace{.625cm}\]
\caption{$\text{Mapping } 1101111110000 \text{ to } 1+2_{i}+1+6_{j}+4$ where $1\leq i \leq 2 \text{ and } 1\leq j\leq6$}
\label{fig:boolean}
\end{figure}

\end{proof}

Following the same essential argument, we have the following generalization of the above observation. The bijection is defined similarly after adding a 1 (instead of a 0) at the very beginning.
\begin{theorem}
The number of ODD colored compositions of $k$ is equal to the sum of the products of the lengths of the runs of 1's in all binary (boolean) strings of length $k$ that start with a 1.
\end{theorem}

\subsection{Related identities}
Let $\{ e(n) \}$ and $\{ o(n) \}$ be counting sequences of the EVEN and ODD colored compositions, respectively. The identity
\begin{equation}\label{eq:1}
e(k+1) = e(k)+o(k)
\end{equation}
can be observed from the Online Encyclopedia of Integer Sequences. With the colored composition interpretation, we can provide a combinatorial proof here.

\begin{proof}[Combinatorial proof of \eqref{eq:1}]
Our proof will follow a bijective map from the collection of EVEN colored compositions of $k+1$ to the disjoint union of EVEN colored compositions of $k$ and ODD colored compositions of $k$. 

Start from an EVEN colored composition of $k+1$, we consider the first part (which is not colored):

\begin{itemize}
    \item If the first part is 1, then remove it to get an ODD colored composition of $k$.
\item If the first part is not 1, then subtract 1 from that part to get an EVEN colored composition of $k$. 
\end{itemize}

To define the inverse map, place a part of size 1 in front of all ODD colored compositions of $k$, which turns them into EVEN colored compositions of $k+1$ that start with a part 1. For EVEN colored compositions of $k$, increase the first part by 1, which turns them into EVEN colored compositions of $k+1$ with the first part $>1$. 
\end{proof}

Let $c_{m,k}(\ell)$ denote the number of $(m,k)$-$n$-colored compositions of $\ell$, the above argument leads to the following generalization.

\begin{theorem}
    For any $k\geq 1$, $\ell \geq 2$, and $1 \leq m \leq \ell -1$, we have
    $$ c_{m,k+1}(\ell + 1) =  c_{m,k+1}(\ell) + c_{m,k}(\ell) . $$
\end{theorem}


\subsection{321-avoiding separable permutations}

Last but not least, we note that (\url{https://oeis.org/A034943}) the EVEN colored compositions of $k$ are equinumerous with the 321-avoiding separable permutations of $[k]$. 

A separable permutation avoids the patterns 2413 and 3142. Among many interesting interpretations of separable permutations, we use the labeled binary tree structure to illustrate our proof. Such a rooted binary tree has leaves labeled with elements of $[k]$, and the descendants of each node form a contiguous subset of $[k]$. Every interior node is either positive, in which case all descendants of the left child are smaller than all descendants of the right child; or negative, in which case all descendants of the left child are greater than all descendants of the right child. 

For 321-avoiding separable permutations, there is the additional restraint that every interior descendant of a negative node must be positive (Figure~\ref{fig:sep}). Note that, however, such trees are not unique as a pair of adjacent positive nodes (there are no adjacent negative nodes by definition) can be replaced by a different pair of nodes using a tree rotation operation.

\begin{figure}[htbp!]
    \centering
    \includegraphics[scale = .55]{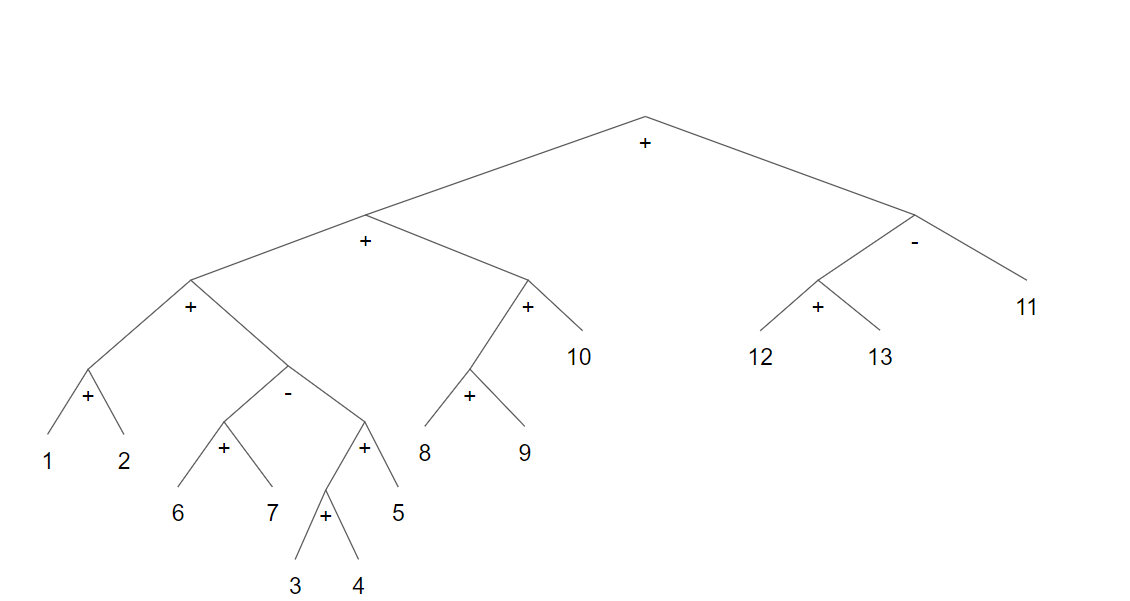}
    \caption{The 321-avoiding separable permutation \\$(1,2,6,7,3,4,5,8,9,10,12,13,11)$ maps to $3+4_2+4+2_2$.}\label{fig:sep}
\end{figure}



\begin{theorem}
    For any positive integer $k$, the EVEN colored compositions of $k$ are equinumerous with the 321-avoiding separable permutations of $[k]$.
\end{theorem}

 \begin{proof}
 To see the bijection between these objects, we start with a binary tree representation for a 321-avoiding separable permutation. For each negative node with $a$ leaf descendants from its left child and $b$ leaf descendants from its right child, we construct a colored part $(a+b-1)_a$. Note that this is possible as $b\geq 1$. Such parts will be even positioned parts that are $n$ colored.

In Figure~\ref{fig:sep}, we have $4_2$ from the leaves $(6,7,3,4,5)$ and $2_2$ from the leaves $(12,13,11)$.

Next, for any $c$ leaves (whose labels are consecutively increasing) induced by the leaves between the negative nodes (including those to the left of the first negative node), we construct a non-colored part $c+1$. Note that these leaf labels are increasing as there are only positive interior nodes involved. In the case that there are no leaves in between two negative nodes (or to the left of the first negative node), we have a part of size $0+1=1$. These will be our odd positioned parts.

In Figure~\ref{fig:sep}, we have $3$ from the leaves $(1,2)$ and $4$ from the leaves $(8,9,10)$.

Consequently, the 321-avoiding separable permutation $$(1,2,6,7,3,4,5,8,9,10,12,13,11)$$ is mapped, via the binary tree representation, to the EVEN colored composition $3+4_2+4+2_2$.

To see the reverse map, we first subtract 1 from each odd positioned part and add 1 to each even positioned part, then allocate the number of leaves accordingly. From  $3+4_2+4+2_2$, we have 2 leaves before the first negative node, a (2,3) split under the first negative node, three leaves after the first negative node, and a (2,1) split under the second negative node.

\end{proof}

\section{Concluding remarks and future work}
\label{sec:con}
We studied the compositions where some parts are $n$-colored and some are not, depending on their positions in the composition. In particular, we find connections from the so-called EVEN and ODD colored compositions to many other combinatorial objects. Bijections are provided for these connections. 

Several natural directions for future study emerge from this work. First, the generating functions for general position $n$-color compositions are provided in Section~2.3, but the connection between these compositions and other combinatorial objects has not been studied. One could also try proving Corollary~\ref{cor1} combinatorially and generalizing it and other results proven in this paper.

Additionally, in Section~3.1, compositions that restrict color 2 are studied, which invites broader questions, such as: How do compositions that restrict other colors or combinations of colors relate combinatorially to positional $n$-color compositions? Are there any interesting properties of positional $n$-color compositions that restrict certain colors?


\end{document}